\documentclass[USletter, 10pt, conference]{ieeeconf}   % Enable this line and disable the 
\IEEEoverridecommandlockouts   

\usepackage{amsmath}
\usepackage{graphicx}
\usepackage{amssymb}
\usepackage{bbm}
\usepackage{easylist}
\usepackage{fp}
\usepackage{siunitx}
\usepackage{dsfont}
\usepackage{graphicx}  
\usepackage{tikz}

\newcommand{\bal}[1] {\ensuremath{\left(\begin{array}{#1}}}
\newcommand{\ear} {\ensuremath{\end{array}\right)}}

\newcommand{\bals}[1] {\ensuremath{\left[\begin{array}{#1}}} % Begin Array Left Square
\newcommand{\ears} {\ensuremath{\end{array} \right] }} % End Array Right Square
\DeclareMathOperator{\diag}{diag}

\newcommand{\one} {\ensuremath{\mathds{1} }} % Vector of ones

\let\leq\leqslant
\let\geq\geqslant

\newcommand{\calE}{\ensuremath{\mathcal{E}}}

\newcommand{\calG}{\ensuremath{\mathcal{G}}}
\newcommand{\calH}{\ensuremath{\mathcal{H}}}

\newcommand{\calL}{\ensuremath{\mathcal{L}}}

\newcommand{\calV}{\ensuremath{\mathcal{V}}}

\newcommand{\bmat}{\begin{matrix}}
\newcommand{\emat}{\end{matrix}}
\newcommand{\bbm}{\begin{bmatrix}}
\newcommand{\ebm}{\end{bmatrix}}
\newcommand{\bpm}{\begin{pmatrix}}
\newcommand{\epm}{\end{pmatrix}}
\newcommand{\bse}{\begin{subequations}}
\newcommand{\ese}{\end{subequations}}
\newcommand{\beq}{\begin{equation}}
\newcommand{\eeq}{\end{equation}}
\newcommand{\ben}{\begin{enumerate}}
\newcommand{\een}{\end{enumerate}}
\newcommand{\beni}{\renewcommand{\labelenumi}{\roman{enumi}.}
\renewcommand{\theenumi}{\roman{enumi}}\begin{enumerate}}
\newcommand{\eeni}{\end{enumerate}\renewcommand{\labelenumi}{\arabic{enumi}.}
\renewcommand{\theenumi}{\arabic{enumi}}}
\newcommand{\bena}{\renewcommand{\labelenumi}{\alpha{enumi}.}
\renewcommand{\theenumi}{\alpha{enumi}}\begin{enumerate}}
\newcommand{\eena}{\end{enumerate}\renewcommand{\labelenumi}{\arabic{enumi}.}
\renewcommand{\theenumi}{\arabic{enumi}}}
\newcommand{\bit}{\begin{itemize}}
\newcommand{\eit}{\end{itemize}}

\newcommand{\R}{\ensuremath{\mathbb R}}

\tikzstyle{vertex}=[circle,fill=black!5,draw=black,minimum size=15mm]

\tikzstyle{terminal vertex} = [circle,fill=black!5,draw=black,minimum size=15mm]
\tikzstyle{edge} = [draw,thick,-]
\tikzstyle{edge2} = [draw,thick,->,red!50]
\tikzstyle{weight} = [font=\small]
\tikzstyle{selected edge} = [draw,line width=2pt,-,red!50]
\tikzstyle{ignored edge} = [draw,line width=5pt,-,black!20]
\tikzstyle{empty vertices} = [circle,fill=white!5]

\newtheorem{theorem}{Theorem}

\newtheorem{lemma}[theorem]{Lemma}

\newtheorem{example}{Example}
\newtheorem{remark}{Remark}

\title{Optimization of the $\mathcal{H}_\infty$-norm of Dynamic Flow Networks }

\author{Alexander Johansson, Jieqiang Wei, Henrik Sandberg, Karl H. Johansson and Jie Chen % <-this %
 \thanks{*This work is supported by Knut and Alice Wallenberg Foundation, Swedish Research Council, Swedish Foundation for Strategic Research and by Hong Kong Research Grants Council (CityU 11200415).}
 \thanks{A. Johansson, J. Wei, H. Sandberg and K.H. Johansson are with the ACCESS Linnaeus Centre, School of Electrical Engineering. 
 KTH Royal Institute of Technology,
 SE-100 44 Stockholm, Sweden. Emails:
         {\tt\small \{jieqiang, kallej\}@kth.se}}
         \thanks{Jie Chen is with the Department of Electronic Engineering, City University of Hong Kong,
Hong Kong, China.}
}

\begin{document}
\maketitle
\thispagestyle{plain}
\pagestyle{plain}

\begin{abstract}                          % Abstract of not more than 200 words.
In this paper, we study the $\calH_\infty \text{- norm}$ of linear systems over graphs, which is used to model distribution networks. In particular, we aim to minimize the $\calH_\infty \text{- norm}$ subject to allocation of the weights on the edges. The optimization problem is formulated with LMI (Linear-Matrix-Inequality) constraints. For distribution networks with one port, i.e., SISO systems, we show that the $\calH_\infty \text{- norm}$ coincides  with the effective resistance between the nodes in the port. Moreover, we derive an upper bound of the $\calH_\infty \text{- norm}$, which is in terms of the algebraic connectivity of the graph on which the distribution network is defined.
\end{abstract}

\iffalse


\fi
 
\section{Introduction}

In this paper we study robustness of a basic model for the dynamics of a distribution
network. Identifying the network with a undirected graph we associate with every
vertex of the graph a state variable corresponding to {\it storage}, and with
every edge a control input variable corresponding to {\it flow}. Furthermore, some of the vertices
serve as terminals where an unknown flow may enter or leave the
network in such a way that the total sum of inflows and outflows is equal to
zero. Many control protocols are designed for a distributed control structure (the control input corresponding to a
given edge only depending on the difference of the state variables of the
adjacent vertices) which will ensure that the state variables associated to all
vertices will converge to the same value, i.e., reach consensus, \cite{Blanchini2000},\cite{Wei2013}.

In this paper, we consider the distribution network controlled by proportional controllers on the edges and study the robustness property with respect to the controller gain, i.e., the edges weights. In particular, we are interested in minimizing the $\calH_\infty \text{- norm}$ by allocating the edge weights. 

The distribution networks can be seen as linear time-invariant port-Hamiltonian systems \cite{vanderschaftmaschkearchive}, \cite{vanderschaftbook}, but also resides in the category of state-space symmetric systems \cite{WILLEMS1976,Ikeda1995,Qiu1996,Yang2001,Nagashio2005}. One important property of the state-space symmetric system is that its $\calH_\infty \text{- norm}$ is attained at the zero frequency \cite{TAN2001}, which is employed to solve the current problem.

The contributions of this paper are: The problem of minimizing the $\calH_\infty \text{- norm}$ of the distribution networks subject to the allocation of the edge weights is formulated and written with LMIs as constraints. Moreover, we give an interpretation of the Riccati inequality which regards definitness of a Laplacian to a graph containing both positive and negative weights on the edges. As a consequence of the interpretation, it is shown for distribution networks with one port, that the $\calH_\infty \text{- norm}$ (or induced  $\mathcal{L}_2 \text{- gain}$) is equal to the effective resistance between the nodes in the port. Then, an upper bound of the $\calH_\infty \text{- norm}$ is derived, which relates to the algebraic connectivity of the graph on which the distribution network is defined. The results in this paper can be relevant when designing robust multi-agent systems. In particular when considering a  malicious attacker, e.g.,  \cite{rai2012}.

The structure of the paper is as follows. Some preliminaries will
be given in Section \ref{s:preli}. The considered class of dynamic flow networks is given and the optimization problem is formulated in Section \ref{ProbForm}. The main results is presented in Section  \ref{distnet}. In Section \ref{Numex} there is a numerical example which demonstrates some results from this paper. Conclusions and future work are given in Section \ref{Concl} and \ref{Future}, respectively.

\textbf{Notation.}  A positive semi-definite (symmetric) matrix $M$ is denoted as $M\succcurlyeq 0$. A positive definite (symmetric) matrix $M$ is denoted as $M\succ 0$. The $i^{\text{th}}$ row of a matrix $M$ is given by $M_{i}$. The element on the $i^{\text{th}}$ row and $j^{\text{th}}$ column of a matrix $M$ is denoted $M_{ij}$. The vectors $e_1,e_2,\ldots,e_n$ denote the canonical basis of $\R^n$, whereas the vectors $\one_n$ and $\mathbf{0}_n$ represent a $n$-dimensional column vector with each entry being $1$ and $0$, respectively. We will omit the subscript $n$ when no confusion arises. The euclidean norm is denoted as $|\cdot |_2$, for a vector $x\in\R^n$, $|x|_2=(x_1^2+\dots+x_n^2)^{\frac{1}{2}}$.

\section{Preliminaries}\label{s:preli}

In this section, we briefly review some essentials about graph theory \cite{Bollobas98}, and give some definitions for robust analysis \cite{zhou1998essentials}.

\subsection{Graph Theory}\label{GraphT}

An undirected graph $\mathcal{G}=(\mathcal{W},\mathcal{V},\mathcal{E})$ consists of a finite set of nodes $\mathcal{V}=\{v_1,...,v_n\}$, a set of edges $\mathcal{E}=\{ \mathcal{E}_1,...,\mathcal{E}_m\}$ which contains unordered pairs of elements of $\calV$, and a set of corresponding edge weights $\mathcal{W}=\{w_1,...,w_m\}$. Graphs with unit weights, i.e., $w_i = 1$, for $i=1,...,m$, are denoted as $\mathcal{G}=(\mathcal{V},\mathcal{E})$. The set of neighbours to node $i$ is 
\begin{equation*}
N_i=\{v_j|(v_i,v_j)\in \mathcal{E} \}.
\end{equation*}
The graph Laplacian  $L\in \mathbb{R}^{n \times n}$ is defined component-wisely as
\begin{equation*}\label{Laplacian}
 L_{ij} =
  \begin{cases}
    \sum_{ j\in N_i} w_{ij} & \quad \text{if }   i=j,  \\
    -w_{ij}  & \quad \text{if } j\in N_i \setminus \{i\},\\
    0 & \quad \text{if } j\notin N_i.
  \end{cases}
 \end{equation*}
Given an orientation for each edge, the incidence matrix $B\in~\R^{n \times m}$ is defined as
\begin{equation*}\label{incid}
B_{ij}=\begin{cases}
    1 & \quad \text{if }   \mathcal{E}_j \ \text{starts in node} \ v_i,  \\
   -1  & \quad \text{if }   \mathcal{E}_j \ \text{ends in node} \ v_i,  \\
    0  & \quad \text{else}. 
  \end{cases}
\end{equation*}
These two matrices are related by $L=BWB^T$, where $W=~\diag(w_1,...,w_m)$. If $W\geq 0$ then the eigenvalues of $L_w$ can be structured as
\begin{equation*}
0=\lambda_1\leq\lambda_2\leq...\leq \lambda_n, 
\end{equation*}
where the eigenvector corresponding to $\lambda_1=~0$ is $\mathds{1}^\top=~[1,...,1]^T$. The second smallest eigenvalue, i.e., $\lambda_2$, is commonly referred to as the \emph{algebraic connectivity} \cite{Fiedler1973} and is a measure of how connected a graph is. Furthermore, if $\mathcal{G}$ is connected, then $\lambda_2>0$. 

If some weights are negative, the Laplacian can be decomposed as 
\begin{equation*}\label{Lapposneg}
L=L_{+}+L_{-}=B_{+}W_{+}B_{+}^T+B_{-}W_{-}B_{-}^T ,
\end{equation*}
where $B_{+}$ and $B_{-}$ are incidence matrices corresponding to the positive and negative sub-graphs, respectively. The matrices $W_{+}$ and $W_{-}$ are the weights of the positive and negative sub graphs, respectively. This decomposition is also used in e.g., \cite{CHEN2016}.

A measure of the connectivity between two nodes in $\mathcal{G}=~(\mathcal{W},\mathcal{V},\mathcal{E})$ is the \emph{effective resistance} \cite{bullo2014}. The effective resistance between the nodes $v_i$ and $v_j$ is defined as  
\begin{equation*}
R_{ij}=(e_i-e_j)^TL^{\dagger}(e_i-e_j),
\end{equation*}
where $L^{\dagger}$ is the Moore-Penrose pseudo inverse of $L$. 

\begin{lemma} [\cite{zelazo2014},Theorem III.3] \label{theodef}
Assume $\mathcal{G}=(\mathcal{W},\mathcal{V},\mathcal{E})$ has one edge with negative weight and the negative edge is $\mathcal{E}_{-}=(u,v)$. Let $\mathcal{G}_{+}$ be the positive sub-graph of $\mathcal{G}=~(\mathcal{W},\mathcal{V},\mathcal{E})$ and assume it is connected. Then $L(\mathcal{G})$ is positive semi-definite if and only if 
\begin{equation*}
|W(\mathcal{E}_{-})|\leq R^{-1}_{uv}(G_{+}),
\end{equation*}
where $W(\mathcal{E}_{-})$ is the negative weight and $R_{uv}$ denotes the effective resistance between node $u$ and $v$.
\end{lemma}

\subsection{$\mathcal{L}_2$-Norm and induced $\mathcal{L}_2$-Gain}

In this subsection, we recall some definitions from robust control. The notations used in this paper are fairly standard and are consistent with \cite{zhou1998essentials}, \cite{RANTZER2015}.
The space of square-integrable signals $f:[0,\infty)\rightarrow \mathbb{R}^n$ is denoted by $L_2[0,\infty)$.
For the linear time-invariant system 
\begin{align}\label{e:linear-sys}
\dot{x} & = Ax+Bu, \\ \nonumber
y & = Cx+Du,
\end{align}
the transfer matrix is $\mathbb{G}(s)=C(sI-A)^{-1}B+D$, which has the impulse response
\begin{equation*}
g(t)~=\calL^{-1}\{\mathbb{G}(s)\}~=~Ce^{At}B\mathbf{1}_+(t)+D\delta(t),
\end{equation*}
where $\delta(t)$ is the unit impulse and $\mathbf{1}_+(t)$ is the unit step defined as 
\begin{align}\nonumber
\mathbf{1}_+(t) = \begin{cases}
1, t\geq 0,\\
0, t<0.
\end{cases}
\end{align}  
If $x(0)=0$, then we have
$
y(t) = \int_{0}^t g(t-\tau) u(\tau) d\tau.
$
Then the induced $\mathcal{L}_2 \text{- gain}$ is defined as
\begin{align}\nonumber
\|g\|_{2-ind} = \sup_{u\in L_2[0, \infty)}\frac{\|y\|_2}{\|u\|_2} = \sup_{u\in L_2[0, \infty)}\frac{\|g*u\|_2}{\|u\|_2},
\end{align}
where 
$\|u(t)\|_2 =\Big(\int_{0}^{\infty}|u(t)|_2^2 dt\Big)^{\frac{1}{2} }.$

This induced $\mathcal{L}_2 \text{- gain}$, i.e., $\|g\|_{2-ind}$ or $\|\mathbb{G}\|_{2-ind}$, is often called the $\calH_\infty \text{- norm}$, denoted as $\|\mathbb{G}\|_\infty$.  It is well-know that 
$\|\mathbb{G}\|_\infty= \sup_{\omega\in\mathbb{R}}\bar{\sigma}\{\mathbb{G}(j\omega)\}$,
where $\bar{\sigma}(A)$ denote the largest singular value of the matrix $A$.

For the system (\ref{e:linear-sys}) with $D=0$, the bounded real lemma \cite{zhou1998essentials} implies that  $\|\mathbb{G}\|_\infty \leq \gamma$ if and only if 
there exists $P~=~P^\top \succ 0$ such that 
\begin{align}\label{flow2feb22}
PA+A^\top P+C^\top C+ \frac{1}{\gamma^2}PBB^\top P \preccurlyeq 0.
\end{align}

\section{Problem Formulation}\label{ProbForm}

We consider the dynamical distribution network defined on a graph $\calG=\{\calV,\calE\}$ with $|\calV|=n$ and $|\calE|=m$. On the vertices, we consider integrators, given as
\begin{align}\label{e:plant}
\dot{x} & =  u, & x,u \in \mathbb{R}^n, \\ \nonumber
z & =  x ,& z \in \mathbb{R}^n.
\end{align}
Here the $i^{\text{th}}$ element of $x$ and $u$, i.e. $x_i$ and $u_i$, are the state
and input variables associated with the $i^{\text{th}}$ vertex of the graph. System \eqref{e:plant} defines a port-Hamiltonian
system \cite{vanderschaftbook}, satisfying the energy-balance
\begin{equation*}
\frac{d}{dt}\frac{|x|_2^2}{2} =  u^T z.
\end{equation*}

As a next step we will extend the dynamical system (\ref{e:plant}) with an external input
$d$ of \emph{inflows and outflows}
\begin{equation}\label{e:plant+disturbance} \nonumber
\begin{aligned}
\dot{x} & =  u + Ed,& \quad & d \in\mathbb{R}^k, \\[2mm]
z & =  x, & \quad &
\end{aligned}
\end{equation}
where $E$ is a $n \times k$ matrix whose columns consist of one element which is $1$ (inflow) and one element $-1$ (outflow), while the rest of the elements are zero. A port is a set of nodes(terminals) to where the external flow which enter and leave the network sums to zero. Thus, $E$ specifies $k$ ports.

To achieve a state consensus, many controllers which provide the flows on the edges of $\calG$ have been proposed, with the following general form
\begin{equation}\label{e:controller general} 
\begin{aligned} 
\dot{\eta}_k &= f_k(\eta_k,\zeta_k), &\\
\mu_k &= g_k(\eta_k,\zeta_k),& \quad k=1,2,\ldots,m
\end{aligned}
\end{equation}
where $\eta_k,\zeta_k,\mu_k$ are respectively the states, input and output of the controller on the $k^{\text{th}}$ edge of $\calG$. Denote the stacked vectors of $\eta_k,\zeta_k,\mu_k$ as $\eta,\zeta,\mu$ respectively. With the controller \eqref{e:controller general}, the state variables $x_i, i=1,2,\ldots,n,$ are controlled by the controller output $\mu_k,k=1,2,\ldots,m,$ in the following manner 
\begin{equation}\label{e:F-connection}\nonumber
 u+BW\mu=0,
\end{equation}
where $B\in\R^{n\times m}$ is the incidence matrix of the digraph $\calG$, and $W$ is the diagonal matrix corresponding to the gain of the controller to the edges. In addition, the controller
is driven by the relative output of the systems \eqref{e:plant} on vertices, i.e
\begin{equation}\label{e:P-connection} \nonumber
 \zeta=B^Tz.
\end{equation}

It is known that, if $d=0$, the state agreement of the system \eqref{e:plant} can be achieved by P-control and PI-control. 
For the P-control, the closed-loop is, 
\begin{equation}\label{flowsys}
\begin{split}
 \dot{x}= & -L_wx+Ed,  \\
 y= & E^Tx,
 \end{split}
\end{equation}
where $y$ is a vector with the components being the state difference at each port.

%We consider a first-order multi-agent system with external inputs. The inputs can represent disturbances, loads, additional controller, etc. The multi-agent system is defined on a graph $\mathcal{G}~=~(\mathcal{V}, \ \mathcal{E}, \ \mathcal{W})$. The graph is assumed to be undirected and only contain positive weights. The LTI state-space representation of the multi-agent system is given as
%
%where $x \in \mathbb{R}^n$, $u \in \mathbb{R}^{p}$ and $y \in \mathbb{R}^{p}$ and $L\in~\mathbb{R}^{n \times n}$ is the graph Laplacian of $\mathcal{G}$. Moreover, we assume $\mathbbm{1}^TE=0$. This system is a Port-Hamiltonian system, see \cite{10}. The nodes belonging to the non-zero elements in each column of $E$ corresponds to a port, i.e., there are as many ports as columns in $E$. 

\begin{example} 
One physical interpretation of the system (\ref{flowsys}) is a basic model of a dynamic flow network, where there are water reservoirs on the nodes and pipes on the edges. The reservoirs are identical cylinders and the pipes are horizontal. The state $x$ is constituted by the water levels in the reservoirs and the pressures are proportional to the water levels. The flow in the pipes are passively driven by pressure difference between the reservoirs. The weights $ \mathcal{W}$ are representing the capacities of the pipes, in terms of diameter and friction. The passive flow from reservoir $i$ to reservoir $j$ is then $q_{ij}=~w_{ij}(x_i-x_j)$. The external input $d$ can e.g. be interpreted as flow in pumps which are distributing water inside the network. The output $y$ is then the difference between water levels of the reservoirs which the pumps are pumping to and the reservoirs which the pumps are pumping from.

Another physical interpretation of the system (\ref{flowsys}) is a mass-damper system, where there are masses on the nodes and dampers on the edges. The damping force is proportional to the relative velocity of the connected masses. The state $x$ is constituted by the momentum of the masses. The weights $ \mathcal{W}$ are representing the damping constants. The input $d$ is representing external forces, which are exposing some masses to push and some masses to pull. The total push is equal to the total pull. The output $y$ is the difference in momentum between the masses which are exposed to push and the masses which are exposed to pull.

There are many other interpretations and applications of the system (\ref{flowsys}). Others are e.g., chemical reaction networks \cite{schaftSIAM} and consensus protocols \cite{Saber2003}.  
\end{example}

In this paper we are interested in the following problem: For a given topology, how to achieve the best robust performance of the system \eqref{flowsys} by arranging the weights on the edges, i.e.,  
\begin{align}\label{oriopt}
\min_{W}  &\|\mathbb{G}\|_\infty  \\ \nonumber
s.t.,    & \sum w_i =c, \ w_i \geq 0,
    \end{align}
where $\mathbb{G}$ is the transfer function of the system (\ref{flowsys}), $W=~\diag(w_1,...,w_m)$ and $w_i$, for $i=1,...,m$, are the weights on the edges. The constant $c$ is the constraint on the sum of all edge weights.

\begin{example}[flow network continued]
For the flow network interpretation of the system \eqref{flowsys}, the optimization problem above is to allocate capacities of the water pipes such the $\calH_\infty \text{- norm}$ of the flow network is minimized. The constant $c$ represents the total capacity of the pipes.  
\end{example}

\medskip 

\section{$\mathcal{H}_\infty$-norm of the distribution network}\label{distnet}

\subsection{Optimization problem reformulated with LMI constraints}\label{LMIre}

We start this subsection by reformulating problem \eqref{oriopt} as an equivalent optimization problem with LMIs as constraints, which can then be efficiently solved numerically using, e.g., with Yalmip \cite{Lofberg2004}. 

\begin{theorem}\label{thm:main}
Consider the system (\ref{flowsys}). If the $\calH_\infty \text{- norm}$ is less than or equal to $\gamma$, then the following LMI is satisfied,
\begin{equation}\label{newLMI}
\begin{bmatrix}
L_w & E \\
E^\top & \gamma I_k
\end{bmatrix}\succcurlyeq  0.  
\end{equation}
\end{theorem}

\begin{proof}
Denote 
\begin{align*}
U^\top=~[\mathds{1}_n, u^\top_2, \ldots, u^\top_n] \ \text{and} \ U_2^\top~=~[u^\top_2, \ldots, u^\top_n],
\end{align*}
for which 
$U L_w U^\top =~\diag(0, \lambda_2, \ldots,\lambda_n)=:~\Lambda$. Denote $\hat{\Lambda}=~\diag(\lambda_2,\ldots,\lambda_n)$. Then the system (\ref{flowsys}) has equal $\calH_\infty \text{- norm}$ as the system
\begin{align*}
\dot{\tilde{x}} & = - \Lambda \tilde{x} + UE d, \\
z & = E^\top U^\top \tilde{x}.
\end{align*}
Notice that the first row of $UE$ is zero, thus the $\calH_\infty \text{- norm}$ of the system (\ref{flowsys}) equals the $\calH_\infty \text{- norm}$ of the system
\begin{equation}\label{e:flow_P_mini}
\begin{aligned}
\dot{\hat{x}} & = - \hat{\Lambda} \hat{x} + U_2E d, \\
z & = E^\top U_2^\top \hat{x}.
\end{aligned}
\end{equation} 
Due to symmetry of the system and by Theorem 6 in \cite{TAN2001}, the $\calH_\infty \text{- norm}$ of the system  (\ref{e:flow_P_mini}) is $\|E^\top U_2^\top \hat{\Lambda}^{-1} U_2E \|_{2}$. The $\calH_\infty \text{- norm}$ of the system \eqref{flowsys} is then less or equal to $\gamma$ if and only if
\begin{equation} \nonumber
\|E^\top U_2^\top \hat{\Lambda}^{-1} U_2E \|_{2} \preccurlyeq \gamma.
\end{equation}
By the property of real symmetric matrix, we can further rewrite the previous constrain as $E^\top U_2^\top \hat{\Lambda}^{-1} U_2E \preccurlyeq \gamma I_{k}$. By Schur complement, we have
\begin{align} \nonumber
\begin{bmatrix}
\hat{\Lambda} & U_2 E \\
E^\top U_2^\top & \gamma I_{k}
\end{bmatrix} \succcurlyeq 0,
\end{align}
which is equivalent to 
\begin{equation} \nonumber
\begin{bmatrix}
\Lambda & UE \\
E^\top U^\top & \gamma I_{k}
\end{bmatrix} \succcurlyeq 0. 
\end{equation}
By pre and post multiplication of matrix $\diag(U^\top, I_k)$ and $\diag(U, I_k)$, respectively, the previous inequality is transformed to
\begin{equation} \nonumber
\begin{bmatrix}
L_w & E \\
E^\top & \gamma I_k
\end{bmatrix} \succcurlyeq 0.
\end{equation}
Then the conclusion follows.
\end{proof}

\begin{remark}
By Theorem \ref{thm:main}, the optimization problem (\ref{oriopt}) is equivalent to 
\begin{equation}\label{optnew}
\begin{aligned}
\min_W & \quad \gamma \\
s.t., & \ \ 
\begin{bmatrix}
L_w & E \\
E^\top & \gamma I_k
\end{bmatrix} \succcurlyeq 0, \\
& \ \ \sum w_i = c, \ w_i \geq 0.
\end{aligned}
\end{equation}
Since the constraints are LMIs, this optimization problem can efficiently be solved with Yalmip. The set up above is used later in Section \ref{Numex}, there the optimal edge weight allocation is determined for the system which is illustrated in Figure \ref{fig:ex1} and the optimal $\calH_\infty \text{- norm}$ is verified in a simulation. 
\end{remark}

In Theorem \ref{thm:main}, we proved that the inequality \eqref{newLMI} is satisfied if the $\calH_\infty \text{- norm}$ is less than or equal to $\gamma$. Moreover, by the bounded real lemma we have that if $\|\mathbb{G}\|_\infty\leq \gamma$, there exists $P~=~P^\top \succ 0$ such that 
\begin{equation}\label{flow2feb}
 -PL_w-L_w^TP+EE^T+ \frac{1}{\gamma^2}PEE^TP \preccurlyeq 0.
\end{equation}
In the next result, we provide one explicit solution to \eqref{flow2feb}.

\begin{theorem}
Consider the system (\ref{flowsys}). If the  $\calH_\infty \text{- norm}$ is less then or equal to $\gamma$, then $P=\gamma I$ is a solution to the Riccati inequality \eqref{flow2feb}. 
\end{theorem}

\begin{proof} 
If $\|\mathbb{G}\|_\infty\leq \gamma$, then by the Schur complement, the LMI (\ref{newLMI}) is equivalent to
\begin{equation}\label{LMI94}
-L_w+\frac{EE^T}{\gamma}\preccurlyeq 0.
\end{equation}
Furthermore, notice that by choosing $P=\gamma I$, the Riccati inequality \eqref{flow2feb} is equivalent to \eqref{LMI94}. Hence the conclusion follows.
%\begin{equation}
% -PL_w-L_w^TP+EE^T+ \frac{1}{\gamma^2}PEE^TP\leq0.
%\end{equation}
%
%Since (\ref{flow2feb}) is equivalent to (\ref{LMI94}) for , it is clear that $P=\gamma^*I$ and $\gamma=\gamma^*$ is a solution to (\ref{flow2feb}). Moreover, the smallest value of $\gamma$ which satisfy (\ref{LMI94}) is the induced $\mathcal{L}_2$-gain. Hence, the restriction $P=\gamma I$ is not conservative. 
\end{proof}

\medskip

%\subsection{Connection between induced $\mathcal{L}_2$-gain and algebraic connectivity}
%
%In this section we discuss the relationship between induced $\mathcal{L}_2$-gain and algebraic connectivity for multi-agent systems in the class (\ref{flowsys}). 
%
%
%\begin{theorem}\label{theo1}
%Consider a multi-agent system, as in (\ref{flowsys}). An upper bound of the induced $\mathcal{L}_2$-gain from the external input $u(t)$ to the output $y(t)$ is 
%\begin{equation}\label{mainresfeb}
% \gamma=\frac{     \bar{\lambda}_{EE^T}    }{\lambda_2},
%\end{equation}
%
%where $\lambda_2$ is the second smallest eigenvalue of the weighted Laplacian $L_w$ and  $\bar{\lambda}_{EE^T}$ is the largest eigenvalue of $EE^T$.
%\end{theorem}
%
%
%
%\begin{proof} Easily shown by using that $Lv=0$ and $E^Tv=0$ and by applying the Courant-Fischer principle (e.g., \cite{mohar1991} and \cite{bellman1997}) on inequality (\ref{LMI94}). 
%
%\end{proof}
%
%The induced $\mathcal{L}_2$-gain is then upper bounded by the algebraic connectivity. Hence, our optimization problem (\ref{oriopt}) is related to the optimization problem in \cite{Fiedler1990}, where the author optimize the algebraic connectivity with respect to edge weights. 
%
%
% Inequality (\ref{LMI94})  will be used in Sections~\ref{Spectral} and \ref{NegativeWeights} when deriving upper bounds of the induced $\mathcal{L}_2$-gain for multi-agent systems in the class (\ref{flowsys}).

%\subsection{Induced $\mathcal{L}_2$-Gain Relates to Effective Resistance}\label{NegativeWeights}

\subsection{Graphical interpretation of the Riccati inequality for dynamic flow networks}\label{ss:graphical interpretation}

In this subsection, we give a graphical interpretation of the Riccati inequality \eqref{newLMI} which is equivalent to \eqref{LMI94} (by Schur complement), for a special type of dynamic flow networks. More precisely, 
we assume that each column of $E$ has exactly two non-zero elements, one is $1$ and the other is $-1$. By this restriction of $E$, it has the structure of an incidence matrix and $EE^T$ is therefore a Laplacian. For $\gamma> 0$, let us define 
\begin{equation}\nonumber
L_\gamma = - \frac{1}{\gamma} EE^T,
\end{equation}
and denote the corresponding graph as $\mathcal{G}_\gamma~=~(\mathcal{W}_\gamma,\mathcal{V}_\gamma,\mathcal{E}_\gamma)$, where $\mathcal{W}_\gamma=\{-\frac{1}{\gamma},...,-\frac{1}{\gamma}\}$ and $\mathcal{V}_\gamma=\mathcal{V}$. The set of edges $\mathcal{E}_\gamma$ is determined by $E$. Recall that $\mathcal{G}=(\mathcal{W},\mathcal{V},\mathcal{E})$ is the graph on which system (\ref{flowsys}) is defined. Moreover, we define $\tilde{L}=L_w+L_\gamma $, which is a Laplacian with both positive and negative weights on the edges. The inequality (\ref{LMI94}) then equals to $\tilde{L} \succcurlyeq 0. $ 
Hence the $\calH_\infty \text{- norm}$ of system \eqref{flowsys} coincides with the largest magnitude of the negative weights $\frac{1}{\gamma}$, which yields a positive definite Laplacian $\tilde{L}$. Notice that it is possible for $\tilde{L}$ to have negative weights.

\begin{example}\label{ex}
The connection between $\mathcal{G}=(\mathcal{W},\mathcal{V},\mathcal{E})$ and $\mathcal{G}_\gamma=(\mathcal{W}_\gamma,\mathcal{V}_\gamma,\mathcal{E}_\gamma)$ is illustrated in this example. Consider a system as in (\ref{flowsys}), where  

\begin{align}
\nonumber L_w & = \begin{bmatrix} 
       w_{12}+w_{13} & -w_{12} & -w_{13} & 0\\[0.3em]
      -w_{12} & w_{12}+w_{24} & 0 & -w_{24} \\[0.3em]
       -w_{13} & 0 &  w_{13}+w_{34} & -w_{34} \\[0.3em]
        0 & -w_{24} & -w_{34} & w_{24}+w_{34} 
     \end{bmatrix}, \\ \label{exampleL}
     \ E^T & =\begin{bmatrix} 
         1 & 0 & 0 & -1\\[0.3em]
         1 & -1 & 0& 0 
         \end{bmatrix}.
       \end{align}

\begin{figure}
\centering
\begin{tikzpicture}[scale=0.7, auto,swap]
  
\foreach \pos/\name in {{(-2,0)/2}, {(2,0)/3}}
        \node[vertex] (\name) at \pos {$\name$};
  
\foreach \pos/\name in {{(0,2.5)/1}, {(0,-2.5)/4}}
     \node[terminal vertex] (\name) at \pos {$\name$};
  
\foreach \pos/\name in {{(0,5)/5}, {(0,-5)/6},{(1.5,4.5)/7},{(-3.5,-2.0)/8}}
     \node[empty vertices] (\name) at \pos {};
  
\foreach \source/ \dest /\weight in {1/2/w_{12}, 1/3/w_{13},2/4/w_{24},3/4/w_{34}}
      \path[edge] (\source) -- node[weight] {$\weight$} (\dest);
  
\foreach \source/ \dest /\weight in {5/1/d_1, 4/6/d_1,7/1/d_2,2/8/d_2}
      \path[edge2]  (\source) -- node[weight] {$\weight$} (\dest);
\end{tikzpicture}

 \caption{The graph on which the system (\ref{exampleL}) is defined. The external inputs to the system, i.e., $d_1$ and $d_2$, is also marked. The output from the system is $y_1=x_1-x_4$ and $y_2=x_1-x_2$. The ports of this system are $(d_1,y_1)$ and $(d_2,y_2)$.}
    \label{fig:ex1}
\end{figure}
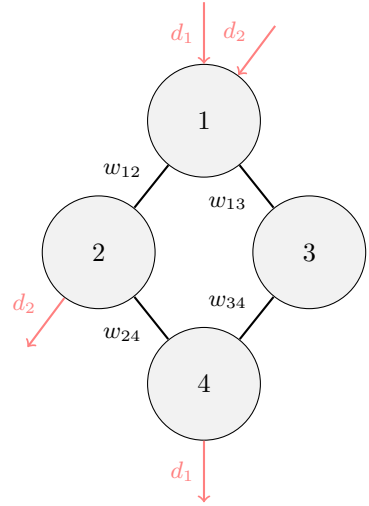

This dynamic flow system is defined on the graph $\mathcal{G}$, which is illustrated in Figure \ref{fig:ex1}. For this system,  the graph $\mathcal{G}_\gamma$ which corresponds to $L_\gamma$ is illustrated in Figure \ref{fig:ex2}. The induced $\mathcal{L}_2 ~\text{- gain}$ from $d=[d_1, d_2]^T$ to $y~=~[x_1-~x_4, x_1-x_2]^T$ is upper bounded by the largest magnitude of the weights $-\frac{1}{\gamma}$ which yields a positive definite $\tilde{L}$.

\begin{figure}
\centering
\begin{tikzpicture}[scale=0.7, auto,swap]
  
  \foreach \pos/\name in {{(-2,0)/2}, {(2,0)/3}}
        \node[vertex] (\name) at \pos {$\name$};
  
  \foreach \pos/\name in {{(0,2.5)/1}, {(0,-2.5)/4}}
     \node[terminal vertex] (\name) at \pos {$\name$};
  
  \foreach \pos/\name in {{(0,5)/5}, {(0,-5)/6},{(1.5,4.5)/7},{(-3.5,-2.0)/8}}
     \node[empty vertices] (\name) at \pos {};
  
  \foreach \source/ \dest /\weight in {1/4/-\frac{1}{\gamma}, 1/2/-\frac{1}{\gamma}}
      \path[selected edge] (\source) -- node[weight] {$\weight$} (\dest);
  
 \end{tikzpicture}
\caption{For the system (\ref{exampleL}), which is illustrated in Figure \ref{fig:ex1}. The graph which corresponds to $L_\gamma=~-\frac{1}{\gamma} EE^T$ is illustrated in this figure. The induced $\mathcal{L}_2 ~\text{- gain}$ from $[d_1, d_2]^T$ to $[y_1, y_2]^T$ is upper bounded by the largest magnitude of $-\frac{1}{\gamma}$, which yields a positive definite $\tilde{L}=L_w+ L_\gamma $. }
 \label{fig:ex2}
\end{figure}
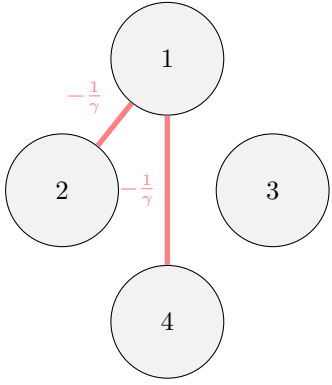
\end{example}

\subsection{Connection between $\calH_\infty \text{- norm}$ and effective resistance}\label{NegativeWeights}

In the previous subsection, we reinterpret the $\calH_\infty \text{- norm}$ of system \eqref{flowsys} in the scenario of the Laplacian $\tilde{L}$ with negative weights. Further conclusions can be drawn from the reasoning above about the definiteness of $\tilde{L}$ if we set the restriction to SISO case, i.e., $E=e_i-e_j$. The $\calH_\infty \text{- norm}$ is then shown to coincide with the effective resistance between node $i$ and $j$. 

%Earlier work provides conditions for when Laplacians, which contains both positive and negative edge weights, are positive definite \cite{zelazo2014}. 
%Their results are more general but in this paper we will only consider the case where there are only one port. i.e., there are only one negative edge. 

%Theorem \ref{theodef} is used to show that the induced $\mathcal{L}_2$-gain equals the effective resistance between node $i$ and $j$ if $E=e_i-e_j$.

\begin{theorem} 
Consider system (\ref{flowsys}) defined on $\mathcal{G}=~(\mathcal{W},\mathcal{V},\mathcal{E})$, which is undirected and only contains positive edge weights. Moreover, assume that there is one port, i.e., $d \in \mathbb{R}$ and $E=e_i-e_j$. Then the induced $\mathcal{L}_2$-gain from $d$ to $y$ is 
\begin{equation*}
\gamma={R_{ij}(L_w)},
\end{equation*}
where $R_{ij}$ denotes the effective resistance between the nodes in the port. Namely, nodes the $i$ and $j$. 
\end{theorem}

\begin{proof}
First note that (\ref{LMI94}) is composed by a positive and a negative graph Laplacian. The negative graph Laplacian has weights $-\frac{1}{\gamma}$. By Lemma \ref{theodef}, the inequality (\ref{LMI94}) is satisfied if and only if
\begin{equation*}
\frac{1}{\gamma}\leq R^{-1}_{ij}(L_w)\iff  \gamma\geq  R_{ij}(L_w).
\end{equation*}
\end{proof}

\section{$\mathcal{H}_\infty$-norm Bounded by Algebraic Connectivity}\label{Spectral}

In Sections \ref{ss:graphical interpretation} and \ref{NegativeWeights}, we showed that the $\calH_\infty \text{- norm}$ has explicit graphical interpretation for a special matrix $E$. In this section, we focus on general matrices $E$. Here we provide one preliminary result which relates the $\calH_\infty \text{- norm}$ of system (\ref{flowsys}) to the algebraic connectivity of the underlying graph. This result can be used if the location of the ports is unknown.

\begin{lemma}\label{theo1}
Consider system (\ref{flowsys}). Then, the $\calH_\infty \text{- norm}$ is bounded by

\begin{equation}\label{mainresfeb} \nonumber
 \gamma=\frac{     \bar{\lambda}_{EE^T}    }{\lambda_2},
\end{equation}
where $\lambda_2$ is the second smallest eigenvalue of the weighted Laplacian $L_w$ and  $\bar{\lambda}_{EE^T}$ is the largest eigenvalue of $EE^T$.
\end{lemma}

\begin{proof} 
The result is shown by using that $L\one=0$ and $E^T\one=0$ and by applying the Courant-Fischer principle (e.g., \cite{mohar1991} and \cite{bellman1960}) on inequality (\ref{LMI94}).
\end{proof}

\begin{remark}
By the previous lemma, maximizing the algebraic connectivity of the graph $\calG$ (with respect to the edge weights) is suboptimal to minimizing the $\calH_\infty \text{- norm}$, for given a $E$. This result can be relevant for design of robust systems when $E$ is unknown. This is e.g. the scenario if a malicious attacker is considered and the attacked nodes are unknown.

\end{remark}

\begin{example}
Consider the dynamic flow network (\ref{flowsys}) and the capacity of the pipes is to be allocated in order to minimize the $\calH_\infty \text{- norm}$ of the system, i.e., the optimization problem (\ref{oriopt}). However, the only information about $E$  which is available is the largest eigenvalue of $EE^T$, i.e., $\bar{\lambda}_{EE^T}$. Since full information about $E$ is not available, it is not possible to minimize the $\calH_\infty \text{- norm}$. Instead, by Lemma \ref{theo1}, we can minimize an upper bound by

\begin{align}\label{exopt} \nonumber
&\max_{W} \lambda_2(L_w)  \\ \nonumber 
& s.t., \sum\omega_i =c.
\end{align}

 This problem of maximizing algebraic connectivity with respect to the edge weights is well-studied, e.g.,  \cite{Fiedler1973} and \cite{Ghosh2008}.

\end{example}

\section{Numerical Example}\label{Numex}

In this section we will demonstrate the results from Section \ref{LMIre}. For this purpose the dynamic flow network in Example \ref{ex} (Figure \ref{fig:ex1}) is used. We aim to allocate capacities of the pipes in order to minimize the $\calH_\infty \text{- norm}$. The optimal allocation of the pipe capacity and the optimal $\calH_\infty \text{- norm}$ is determined numerically in Yalmip and the optimization set up (\ref{optnew}). 

The total pipe capacity is set to $c=1$. The optimal allocation of the pipe capacity is $w_{12}^*=~0.6$, $w_{24}^*=0.4$, $w_{13}^*=w_{34}^*=0$ and the optimal $\calH_\infty \text{- norm}$ is $\gamma^*=5$. The flow network with optimally allocated pipe capacities is seen in Figure \ref{fig:ex3}.

\begin{figure}
\centering
\begin{tikzpicture}[scale=0.7, auto,swap]
  
\foreach \pos/\name in {{(-2,0)/2}, {(2,0)/3}}
        \node[vertex] (\name) at \pos {$\name$};
  
\foreach \pos/\name in {{(0,2.5)/1}, {(0,-2.5)/4}}
     \node[terminal vertex] (\name) at \pos {$\name$};
  
\foreach \pos/\name in {{(0,5)/5}, {(0,-5)/6},{(1.5,4.5)/7},{(-3.5,-2.0)/8}}
     \node[empty vertices] (\name) at \pos {};
  
\foreach \source/ \dest /\weight in {1/2/0.6 ,2/4/0.4}
      \path[edge] (\source) -- node[weight] {$\weight$} (\dest);
  
\foreach \source/ \dest /\weight in {5/1/d_1, 4/6/d_1,7/1/d_2,2/8/d_2}
      \path[edge2]  (\source) -- node[weight] {$\weight$} (\dest);
\end{tikzpicture}

 \caption{Flow network (\ref{exampleL}) with the optimal allocated pipe capacities. I.e., $w_{12}=~0.6$,  $w_{24}=0.4$, $w_{13}=w_{34}=0$. The $\calH_\infty \text{- norm}$ corresponding to this allocation is   $\gamma^*=5$.}
    \label{fig:ex3}
\end{figure}
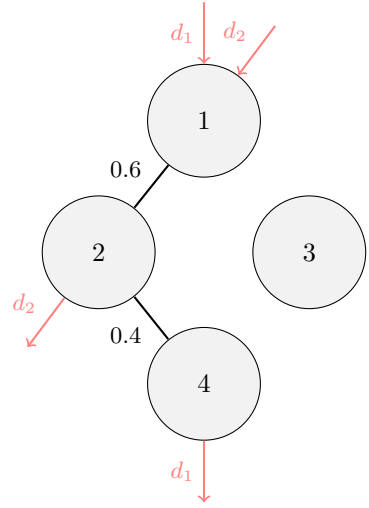

Next, we induce an input to the system in order to verify the $\calH_\infty \text{- norm}$. The input is

\begin{equation}\nonumber
d(t)=\begin{cases}
    [1,0]^T       & \quad \text{if }0\leq t < 1, \\
    [1,1]^T  & \quad \text{if } 1\leq t < 2,\\
    [0,0]^T  & \quad \text{if }2\leq t. \\
  \end{cases}
\end{equation}

In Figure \ref{fig:sim}, the $\mathcal{L}_2$-norm of the output, i.e. $||y(t)||_2$, is seen together with the $\mathcal{L}_2$-norm of the input, scaled with the optimal $\calH_\infty \text{- norm}$, i.e. $ \gamma^*||d(t)||_2$. In the figure it is seen that $||y(t)||_2 \leq \gamma^*||d(t)||_2$, hence the $\calH_\infty \text{- norm}$ is verified.

\begin{figure}
    \centering
    \includegraphics[scale=0.42]{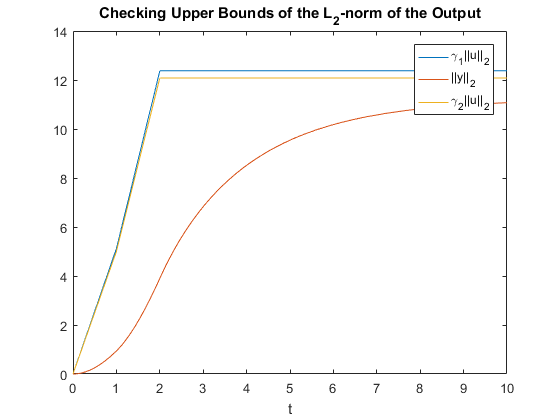}
    \caption{The $\mathcal{L}_2$-norm of the output, i.e., $||y(t)||_2$, is seen together with the $\mathcal{L}_2$-norm of the input scaled with the induced $\mathcal{L}_2 \text{- gain}$, i.e., $\gamma^* ||d(t)||_2$.  }
    \label{fig:sim}
\end{figure}

\section{Conclusions}\label{Concl}

For the dynamic flow networks which we have considered, we have derived an optimization set up with LMIs as constraints, which minimizes the $\calH_\infty \text{- norm}$ with respect to the allocation of the capacity of the pipes. Moreover, for the flow networks, we have interpreted the Riccati inequality as a definiteness criterion of a Laplacian to a graph containing both positive and negative weights on the edges. For flow networks which are SISO, more precisely, $E=e_i-e_j$, we have shown that the $\calH_\infty$-norm coincides with the effective resistance between node $i$ and  node $j$. Moreover, we have derived an upper bound of the induced  $\calH_\infty$-norm of the flow networks. This upper bound relates to the algebraic connectivity on which the flow network is defined. This upper bound can be relevant when full information about the input matrix, i.e $E$, is not available. Then, the capacities of the pipes can be allocated to get a suboptimal solution which bounds the $\calH_\infty \text{- norm}$.

\section{Future work}\label{Future}

A related future topic is the problem of minimizing the $\calH_\infty \text{- norm}$ of dynamic flow networks with respect to topology, more precisely, a limited amount of edges is to be allocated in a graph with fixed vertices. Another future topic is to consider a fixed graph (both topology and weights), but consider saturation of the flow on the edges. The problem is then to minimize the induced $\calL_2$-gain with respect to allocation of the saturation limits.

%\begin{ack}                               % Place acknowledgements
%Thanks  % here.
%\end{ack}

%bibliographystyle{plain}        % Include this if you use bibtex 
%\bibliography{autosam}           % and a bib file to produce the 
                                 % bibliography (preferred). The
                                 % correct style is generated by
                                 % Elsevier at the time of printing.

%\begin{thebibliography}{99}     % Otherwise use the  
                                 % thebibliography environment.
                                 % Insert the full references here.
                                 % See a recent issue of Automatica 
                                 % for the style.

\bibliographystyle{plain} % use IEEEtran.bst style
\bibliography{ref}% name of bib file

%\appendix
%\section{}    % Each appendix must have a short title.
%\section{}         % Sections and subsections are supported  

                                        % in the appendices.
\end{document}